\documentclass[a4paper,english,reqno]{amsart}

\usepackage{amssymb,amsthm, amsmath,amsfonts}
\usepackage{bbm}
\usepackage{graphicx}
\graphicspath{{figs/}}
\usepackage{subfigure}
\usepackage{microtype}
\usepackage{booktabs}
\usepackage{enumerate}
\usepackage{pgf}
\usepackage[latin1]{inputenc}
\usepackage{verbatim}
\usepackage{cite}
\usepackage{url}
\usepackage{xcolor}

\usepackage[cmtip,all]{xy}
\xyoption{arc}

\usepackage{pgf,tikz}
\usetikzlibrary{arrows,shapes,calc,positioning}
\usepackage{manfnt}
\usepackage{tikz-cd}
\usepackage{url}
\usepackage{microtype}

\usepackage[font=footnotesize]{caption}

\usepackage{xcolor}
\usepackage{hyperref}

\DeclareMathOperator{\reals}{\mathbb{R}}

\newenvironment{myproof}[2] {\emph{Proof of {#1} {#2}.}}{\hfill$\square$}

\newcommand{\R}{\mathbb{R}}

\renewcommand{\H}{\mathbb{H}}

\newcommand{\N}{\mathbb{N}} 

\newtheorem{theorem}{Theorem}
\newtheorem{lemma}{Lemma}
\newtheorem{proposition}{Proposition}
\newtheorem{corollary}{Corollary}

\theoremstyle{definition}
\newtheorem{definition}[theorem]{Definition}

\newtheorem*{theoremA*}{Main Theorem}
\newtheorem*{theoremB*}{Main Theorem for a Flute}

\numberwithin{equation}{section}

\begin{document}
	
	\title[On the non-expansiveness of the geodesic flow]{On the non-expansiveness of the geodesic flow on surfaces with cusps}
	
	
	%
	

	
	\author[ S. Burniol]{ Sergi Burniol Clotet}	
	\address{IMERL, CMAT, Universidad de la República, IFUMI, IRL CNRS 2030, Uruguay}
	
	\email{sergi.burniol@gmail.com}
	
		\author[F. Dal'Bo]{Françoise Dal'Bo}	
	\address{IRMAR, Universit\'e de Rennes, France and IFUMI, IRL CNRS 2030
		}
	\email{francoise.dalbo@univ-rennes.fr}
	
	
	\date{23 March 2026}

	\dedicatory{}
	
	
	\begin{abstract} 
		
		We exhibit orbits of the geodesic flow on a hyperbolic surface with at least one cusp such that every tubular neighborhood contains uncountably many distinct geodesic flow orbits. The proof relies on new phenomena, namely the existence of strong stable sets in the dynamical sense that do not coincide with the stable horocycles. When the surface has finite volume, this phenomenon is typical.
	\end{abstract}
	
	\maketitle

	\section{Introduction}
	
	The notion of expansiveness for flows on compact metric spaces was introduced in the 1970s \cite{BowenWalters1972} (see also \cite{KatokHasselblatt1995} for a slightly different formulation). Its origins are closely tied to the study of hyperbolic dynamical systems. In particular, it is well known that Anosov flows on compact manifolds are expansive \cite{Anosov1}. A classical example of an Anosov flow--and hence of an expansive flow--is the geodesic flow on a compact Riemannian manifold of negative curvature. The goal of this paper is to investigate the noncompact setting.


We consider the geodesic flow $g_t$ on the unit tangent bundle $T^1S$ of a hyperbolic surface $S$ with at least one cusp. For instance, $S$ may be a noncompact surface of finite volume. We endow $T^1S$ with the Sasaki distance $d_{Sa}$ induced by the hyperbolic metric of $S$.

We say that a vector $u\in T^1S$ is \emph{cusp-recurrent} if the associated geodesic ray $u(\reals^+)$ in $S$ does not diverge through a cusp and intersects closed horocycles of arbitrarily small length. Our main result shows that, in contrast with the compact setting, expansiveness fails in an unexpected way.

\begin{theorem}\label{thm1}
	Let $S$ be a nonelementary hyperbolic surface with at least one cusp. Let $u\in T^1S$ be cusp-recurrent. Then for every  $\delta>0$ there exist uncountably many vectors $v\in T^1S$ with distinct $g_t$-orbits such that 
	\[
	d_{Sa}(g_t(u),g_t (v)) <\delta, \quad \forall t \in \reals.
	\]
\end{theorem}
	
	The proof is based on the discrepancy between stable horocycles and stable sets. For $u\in T^1 S$, the stable horocycle $H^{ss}u$ is defined as the projection of the stable horocycle of any lift of $u$ to $T^1 \H$, whereas the stable set $W^{ss} u$ is the set
	$$
	W^{ss}u=\{ v\in T^1S:~ \lim_{t\to +\infty} d_{Sa}(g_tu,g_tv)=0 \}.
	$$ 
	
	\begin{theorem}\label{thm2}
		Let $S$ be a hyperbolic surface with at least one cusp. Let $u\in T^1S$ be cusp-recurrent. Then 
		\begin{enumerate}
			\item $H^{ss} u\subsetneq W^{ss}u$,
			\item the set $W^{ss}u$ is an uncountable union of stable horocycles $H^{ss} v_i$.
		\end{enumerate} 
	\end{theorem}
	
	In fact, we exhibit vectors $v \in W^{ss}u$ such that any lifts of the geodesic rays of $u$ and $v$ have distinct endpoints in the ideal boundary.
	
	In \cite{Bellis2018}, Bellis shows that if the infimum of the injectivity radius of $S$ along the geodesic ray $u(\reals^+)$ is positive then $H^{ss}u$ and $W^{ss}u$ coincide. He also asserts that, if $u(\reals^+)$ meets closed geodesics of arbitrarily small length, then these sets are different (see \cite{BurniolDalboHerrero2026} for a complete proof of this fact). Note that in this case the surface is necessarily geometrically infinite. 
	
	Stable horocycles are the orbits of the horocyclic flow on $T^1S$, and have been extensively studied since the seminal work of Hedlund~\cite{Hedlund36}; see also~\cite{dalbotrajectories} for an introduction. When $S$ has finite volume, as an application of the ergodicity of the volume measure under the geodesic flow, we obtain that $H^{ss}u$ does not coincide with $W^{ss}u$ for typical $u\in T^1S$.	
	\begin{theorem}\label{thm3}
		Let $S$ be a noncompact finite volume surface. Then, for almost every vector $u\in T^1 S$, the strong stable set $W^{ss}u$ is an uncountable union of stable horocycles.
	\end{theorem}
	
	When $S$ is geometrically finite, closures of stable horocycles are classified. In particular, if $S$ has finite volume, stable horocycles are either dense or periodic. In contrast, in the geometrically infinite case, there exist stable horocycles with wild closures \cite{Bellis2018,FarreLandesbergMinsky,FarreLandesbergMinsky24,DalboFarreLandesbergMinsky2026}. Although $H^{ss}u$ and $W^{ss}u$ may be different, it is known that their closures always coincide \cite[Corollary 8.3]{FarreLandesbergMinsky}.
	
	The article is self-contained. In Section 2 we recall the essential facts in hyperbolic geometry that will be used. In Section 3 we introduce the notion of winding around a closed horocycle. Section 4 is devoted to the proof of Theorem \ref{thm2}. We discuss abundance of cusp-recurrent vectors and prove Theorem \ref{thm3} in Section 5. Finally, we discuss non-expansiveness and show Theorem \ref{thm1} in Section 6.
	

	


\medskip

	\textbf{Acknowledgment:} We thank Felipe Riquelme for fruitful discussions and Viveka Erlandsson for pointing out the argument of Proposition \ref{nonexpansiveness}.
		
	\section{Preliminaries}
	
	We denote by $\mathbb H$ the hyperbolic plane endowed with its hyperbolic distance $d$. 
	Throughout the paper we use the upper half-plane model $\mathbb H=\{z=x+iy\in\mathbb C : y>0\} $
	equipped with the metric
	\[
	ds^2=\frac{dx^2+dy^2}{y^2}.
	\]
	The Riemannian distance induced by this metric is denoted by $d$. We denote the unit tangent bundle of $\H$ by $T^1\mathbb H$, and the basepoint projection by $\pi:T^1\mathbb H\to\mathbb H$.
	
	\medskip
	
	\subsection{Horocycles and boundary at infinity}
	
	The boundary at infinity of $\mathbb H$ is $\partial\mathbb H=\mathbb R\cup\{\infty\}$.
	Every geodesic ray in $\mathbb H$ converges to a unique point of $\partial\mathbb H$, and two rays define the same point at infinity if and only if they remain at bounded distance from each other.
	
	\medskip
	
	Let $\xi\in\partial\mathbb H$, and let $(c(t))_{t\ge0}$ be a geodesic ray converging to $\xi$. 
	The \emph{Busemann function} based at $\xi$ is defined by
	\[
	B_\xi(x,y)
	=
	\lim_{t\to+\infty}
	\big(
	d(x,c(t)) - d(y,c(t))
	\big),
	\qquad x,y\in\mathbb H.
	\]
	This limit exists and does not depend on the choice of the geodesic ray $c$ asymptotic to $\xi$. 
	For fixed $x\in\mathbb H$, the level sets of the map $y\mapsto B_\xi(y,x)$ are called \emph{horocycles} centered at $\xi$. Geometrically, horocycles are circles tangent to the real line or horizontal lines.
	
	\medskip
	
	For $u\in T^1\mathbb H$, we denote by $u(+\infty)\in\partial\mathbb H$ its forward endpoint. 
	The \emph{stable horocycle} through $u$ is the subset
	\[
	H^{ss} u
	=
	\{ v\in T^1\mathbb H : v(+\infty)=u(+\infty) 
	\text{ and } 
	B_{u(+\infty)}(\pi(v),\pi(u))=0 \}. 
	\] We will also denote 
	\[
	H^{ws} u
	=
	\{ v\in T^1\mathbb H : v(+\infty)=u(+\infty)  \}. 
	\]
	
	\medskip
	
	\subsection{Distance on the unit tangent bundle}

	
	We consider the distance $d_1$ on the unit tangent bundle of $T^1\H$ defined as follows \cite{ballmannlectures}:
	$$ d_1(\tilde v,\tilde w):=d(\tilde v(0),\tilde w(0))+d(\tilde v(1),\tilde w(1)), \quad \tilde v, \tilde w \in T^1 \H.$$
	
	\begin{proposition}
		The distance $d_1$ is equivalent to the Sasaki distance on the unit tangent bundle $T^1 \H$. 
	\end{proposition}

	\medskip
	
	\subsection{Fuchsian groups}
	Throughout the paper, $\Gamma$ denotes a torsion-free Fuchsian group acting properly discontinuously by isometries on $\mathbb H$. 
	We consider the quotient surface $S=\Gamma\backslash\mathbb H$
	and its unit tangent bundle
	$	T^1S=\Gamma\backslash T^1\mathbb H$.
	We use a tilde to denote lifts to the universal cover: given $u\in T^1S$, we write $\tilde u\in T^1\mathbb H$ for a unit tangent vector projecting onto $u$ under the canonical projection $T^1\mathbb H\to T^1S$.
	
	The sets $H^{ss}u$ and $H^{ws}u$ are $\Gamma$-equivariant and therefore descend to well-defined subsets in $T^1S$, which we still denote by $H^{ss}u$ and $H^{ws}u$.
	
	On the unit tangent bundle $T^1 S$ of a hyperbolic surface $S=\Gamma\backslash\mathbb H $, the distance $d_1$ descends as follows: 
	
	$$d_1(u,v)=\inf_{\gamma\in \Gamma}\{d_1(\tilde u,\gamma\tilde v)\},$$
	where $\tilde u$ and $\tilde v$ are lifts of $u$ and $v$ to $T^1\H$. This distance is again equivalent to the Sasaki metric on $T^1S$.
	
	The strong-stable set of a vector $u\in T^1S$ is defined as
	$$
	W^{ss}u=\{ v\in T^1S:~ \lim_{t\to +\infty} d_1(g_tu,g_tv)=0 \}.
	$$ 
	We remark that, since the distance $d_1$ is equivalent to the Sasaki distance, they give the same notion of strong-stable set.

	\subsection{Cusps and parabolic isometries}
	
	Let $\Gamma$ be a Fuchsian group and let $p\in \Gamma$ be a parabolic isometry, that is, an isometry fixing a unique point $x_p$ in $\partial \H$. The stabilizer $\Gamma_p$ of $x_p$ in $\Gamma$ is a cyclic group. There exist a horoball $B_p$ centered at $x_p$ such that the $\Gamma_p \backslash B_p$ projects isometrically to $S=\Gamma \backslash \H$. By definition, this projection is called a cusp.
	
	Each horocycle centered at $x_p$ projects to a closed horocycle in $S$, and every closed horocycle in $S$ arises in this way.
	
	\subsection{Some technical lemmas}
	
	The following two lemmas can be proved by elementary computations using the formula for the hyperbolic distance \cite[Theorem 7.2.1]{Beardon1983}.
	\begin{lemma}\label{decreasing}
		Let $u,v\in T^1\H$ such that $v\in H^{ws} u$. Then $d(v(t),u(t))$ is non-increasing in $t\in \reals$.
		
	\end{lemma}

	\begin{lemma}\label{distance_horo}
		Let $u,v\in T^1\H$ such that $v\in H^{ss} u$. Then for all $t\ge0$
		\[
		\sinh\left(\frac{d(v(t),u(t))}{2}\right) = \sinh\left(\frac{d(v(0),u(0))}{2}\right)e^{-t}.
		\]
	\end{lemma}



%

\par 

\section{Winding around a closed horocycle}

\begin{definition}
	Let $(\tilde H,p)$ be a pair where $\tilde H$ is a horocycle in $\mathbb{H}$ and $p$ be a parabolic isometry preserving $\tilde H$. The \emph{translation length} $\ell(\tilde H,p)$ of $(\tilde H,p)$ is defined as the length of the horocyclic arc between any $x\in \tilde H$ and $px$.
	The \emph{positive orientation} of $\tilde H$ relatively to $p$ is the orientation induced by any arc-length parametrization $s\mapsto \tilde H(s)$ of $\tilde H$ satisfying \[
	p \tilde H(s) = \tilde H (s+\ell (\tilde H,p)), \quad \forall s\in \reals.
	\]
\end{definition}


\begin{definition}A vector $\tilde u\in T^1 \H$ is \emph{tangent to the oriented pair} $(\tilde H, p)$ if:
	\begin{itemize}
		\item the geodesic ray $\tilde u (\mathbb{R}^+)$ is tangent to $\tilde H$ at a point $\tilde u(t_0)$,
		\item if $s\mapsto \tilde H(s)$ is an arc-length parametrization of $\tilde H$ with positive orientation and $\tilde u(t_0)= \tilde H (s_0)$, then
		\[
		\frac{\mathrm{d}\tilde u}{\mathrm{d} t}(t_0) = \frac{\mathrm{d}\tilde H}{\mathrm{d} s}(s_0).
		\]
	\end{itemize}
\end{definition}

\begin{definition}
	Let $\tilde u\in T^1 \H$ be tangent to an oriented pair $(\tilde H, p)$. 	
	We define the \emph{winding of $\tilde u$ around the pair} $(\tilde H, p)$ (see Figure \ref{pic1}), denoted by $\mathrm{Wind}_{(\tilde H, p)}(\tilde{u})\in T^{1}\mathbb{H}^{2}$, as the unique vector satisfying
	\begin{itemize}
		\item $\mathrm{Wind}_{(\tilde H, p)}(\tilde {u})(0)=\tilde{u}(0)$,
		\item $\mathrm{Wind}_{(\tilde H, p)}(\tilde{u})(+\infty)= p \, \tilde{u}(+\infty) $.
	\end{itemize}
\end{definition}

 

	\begin{figure}[h]
	\centering
	\includegraphics[scale=1.5]{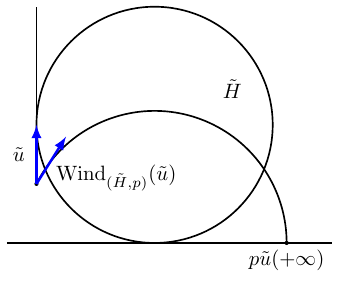}
	\caption{Winding of $\tilde u$ around the pair $(\tilde H, p)$.} 
	\label{pic1}
\end{figure}

\begin{definition}\label{def_winding_time}

The \emph{winding time} associated to $\operatorname{Wind}_{(\tilde H, p)}(\tilde u)$ is the real number $\tau_{\tilde H , p,\tilde u}$ defined by
$$
\tau_{\tilde H, p,\tilde u}=B_{\tilde u(+\infty)}(p^{-1}\tilde u(0), \tilde u(0)).
$$

\end{definition}


 \begin{proposition}[bound of the winding time]\label{bound_wind}
 	Let $\tilde u\in T^1 \H$ be tangent to an oriented pair $(\tilde H, p)$.  	
 	Then
 	\[
 	\left| \tau_{\tilde H, p,\tilde u} \right| \leq \ell(\tilde H, p).
 	\]
 	
 \end{proposition}
 \begin{proof}
 	Let $q$ be the tangency point between $\tilde u (\mathbb{R}^+)$ and $\tilde H$.
 	As $q$ belongs to $\tilde u(\R^+)$, we have
 	$$
 	B_{\tilde u(+\infty)}(q,\tilde u(0))=-d(q,\tilde u(0)).
 	$$
 	In addition, as $q\in \tilde H$ we know that $d(p^{-1}q,q)\le \ell(\tilde H, p)$. Thus
 	
 	\begin{align*}
 		B_{\tilde u(+\infty)}(p^{-1}\tilde u(0),\tilde u(0)) 
 		& =  B_{\tilde u(+\infty)}(p^{-1}\tilde u(0),p^{-1}q)+  
 		B_{\tilde u(+\infty)}(p^{-1}q,q) +B_{\tilde u(+\infty)}(q,\tilde u(0)) \\
 		& \leq  d(p^{-1}\tilde u(0),p^{-1}q)+\ell(\tilde H,p)-d(q,\tilde u(0))\\
 		& = \ell(\tilde H, p).
 	\end{align*}
 	
 	For the other inequality, we observe that $p \tilde u (+\infty)$ is between $\tilde u (+ \infty)$ and the fixed point $x_p$ of $p$, so the geodesic ray $[\tilde u(0), p \tilde u(+\infty))$ intersects $\tilde H$ two times. Let $z$ be the first point of intersection between $[\tilde u(0), p \tilde u(+\infty))$ and $\tilde H$ and denote $q'=p^{-1} z$.  	
 	Then $B_{\tilde u(+\infty)}(p^{-1}\tilde u(0),q')=d(p^{-1}\tilde u(0),q')$ and $d(q',p q')\le \ell(\tilde H,p)$. 
 	
 	Therefore 	
 	\begin{align*}
 		B_{\tilde u(+\infty)}(p^{-1}\tilde u(0),\tilde u(0)) &= B_{\tilde u(+\infty)}(p^{-1}\tilde u(0),q')+B_{\tilde u(+\infty)}(q',p q')+B_{\tilde u(+\infty) }(p q',\tilde u(0)) \\
 		& \geq d(p^{-1}\tilde u(0),q')-d(q',p q')-d(p q',\tilde u(0))\\
 		&\ge - \ell (\tilde H, p).
 	\end{align*}
 \end{proof}
 
 
 
 \begin{proposition}[Key Proposition]\label{key_prop} 
 	Let $\tilde u\in T^1 \H$ be tangent to an oriented pair $(\tilde H, p)$.  	
 	Then, for every $t\ge 0$,
 	\begin{align*}
 		\min \{ d_1\!\left(g_{t+\tau_{\tilde H, p ,\tilde u}}\bigl(\operatorname{Wind}_{(\tilde H, p )}(\tilde u)\bigr),\, g_t\tilde u\right),
 		d_1\!\left(g_{t+\tau_{\tilde H, p ,\tilde u}}\bigl(\operatorname{Wind}_{(\tilde H, p )}(\tilde u)\bigr),\, p g_t\tilde u\right)
 		\}
 		\\ \le 12 \ell(\tilde H, p).
 	\end{align*}
 	
 \end{proposition}
 \begin{proof}
 	For simplicity, we denote $\tilde v=\mathrm{Wind}_{(\tilde H, p )}(\tilde u)$, $\tau=\tau_{\tilde H, p ,\tilde u}$ and $\ell = \ell (\tilde H, p)$.
 	Up to applying an isometry of $\mathbb{H}$, we can assume that $\tilde H$ is centered at $\infty$, $\tilde u(0)=i$ and the isometry $p$ is of the form $p(z)=z+\lambda$ with $\lambda>0$. In this situation, $\tilde u (+\infty)$ must be a positive number.
 	
 	Let $q=\tilde u(t_1)$, $t_1\ge 0$, the point where $\tilde H$ and the geodesic ray $\tilde u (\mathbb{R}^+)$ are tangent. We observe that $p\tilde u (+\infty )$ lies on the right of $\tilde u (+\infty)$. The geodesic ray $\tilde v (\mathbb{R}^+)$ is tangent to a unique horocycle $\tilde H'$ centered at infinity, and $\tilde H'$ lies above $\tilde H$. Let $q' = \tilde v (t_1')$, $t_1'\ge 0$, be the tangency point. The first step is to control the distance between $q$ and $q'$.
 	
 	Write
 	\[
 	q= q_1 + e^{b} i, \quad q'=q_1' +e^ {b'} i, \quad q_1,q_1', b, b'\in \reals.
 	\]
 	Then
 	\[
 	d(q,q')\le d(q_1 + e^{b} i, q_1' + e^{b} i) + d(q_1' + e^{b} i, q_1' + e^{b'} i) \le \frac{|q_1'-q_1|}{e^b} + |b'-b|.
 	\]
 	 We also notice that $\ell= \lambda e^{-b}$. Comparing the radii and the centers of the euclidean half-circles given by the geodesics generated by $\tilde u$ and $\tilde v$ (see Figure \ref{pic2}), we obtain the following estimates:
 	 \begin{enumerate}
 	 	\item $0\le 2e^{b'} -2 e^b <\lambda$,\label{eq1}
 	 	\item $0\le q_1'-q_1 <\lambda$.\label{eq2}
 	 \end{enumerate}

	\begin{figure}[h]
	\centering
	\includegraphics[scale=1.5]{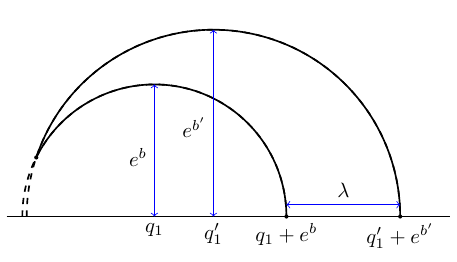}
	\caption{Radii and centers of the half-circles associated to $\tilde u$ and $\tilde v$.} 
	\label{pic2}
\end{figure}
  
  From (\ref{eq1}), we obtain $0\le b'-b < \log (1 +\ell/2) \le \ell/2$, and (\ref{eq2}) implies $0\le (q_1'-q_1)e^{-b}< \ell$. Hence,
  \[
  d(q,q')\le \frac{3}{2} \ell.
  \]
 	
%
 	
 	Moreover, by a triangular inequality,
 	\[
 	|t_1 - t'_1 | \le d(q,q' ) \le \frac{3}{2} \ell.
 	\]

 	We now bound the distance between the basepoints.
 	
 	\noindent\textbf{Case 1.}  If $0\le t\le t_1$, then 
 	\[
 	d\bigl(\tilde v(t),\tilde u(t)\bigr)\le 3\ell.
 	\] 

 Since the distance in $\mathbb{H}$ between two geodesic rays starting at the same point is increasing, it follows 
 \[
  	\forall\,0\le t\le t_1,\qquad
  	d\bigl(\tilde v(t),\tilde u(t)\bigr)\le d\bigl(\tilde v(t_1),\tilde u(t_1)\bigr).
  	\]
Using the bounds established above, one obtains
 	\[
 	d\bigl(\tilde v(t_1),\tilde u(t_1)\bigr) \le 
 	|t_1 - t'_1 | + d\bigl(\tilde v(t'_1),\tilde u(t_1)\bigr) = d\bigl(q',q\bigr) \le
 	3\ell,
  	\]
which implies the statement.
 	
%
 	
 	\smallskip
 	\noindent\textbf{Case 2:} If $ t\ge  t_1$, then 
 	\[
 	d\bigl(\tilde v(t), p \tilde u(t)\bigr)\le 4\ell.
 	\] 
 	
 	Since $\tilde v$ and $p \tilde u$ have the same point at infinity, the distance $d\bigl(\tilde v(t), p \tilde u(t)\bigr)$ is decreasing, and hence for $t\ge t_1$,
 	\[
 	d\bigl(\tilde v(t), p\tilde u(t)\bigr)
 	\le d\bigl(\tilde v(t_1),p \tilde u(t_1)\bigr).
 	\]
 	Moreover,
 	\[
 	d\bigl(\tilde v(t_1), p \tilde u(t_1)\bigr)
 	\le |t_1-t_1'| + d\bigl(\tilde v(t_1'), p \tilde u(t_1) \bigr) \le |t_1-t_1'| + d(q', q) + d(q,pq)
 	\]
 	By the bounds proved above, the first two terms on the right hand side are bounded by $3\ell/2$. The last term is bounded by $\ell$ by definition.
 	We conclude that, for $t\ge t_1$, $d\bigl(\tilde v(t), p \tilde u(t)\bigr)  \le 4\ell$. 
 	
 	\smallskip
 	
 Let us now compute the bounds for the distance $d_1$. Recall that 
 \[
d_1\!\left(g_{t} \tilde v ,\, g_t\tilde u\right) = d(\tilde v (t), \tilde u (t)) + d(\tilde v (t+1), \tilde u (t+1)). 
 \]
 We can distinguish three scenarios:
 
 \begin{enumerate}[(a)]
 	\item If $t \le t_1 -1 $, by the first case, we have that both
 	\[
 	d(\tilde v(t), \tilde u(t))
 	\quad \text{and} \quad
 	d(\tilde v(t+1), \tilde u(t+1))
 	\]
 	are bounded by $3\ell$. Then
 	\[
 	d_1(g_{t}\tilde v, g_t \tilde u) \le 6\ell.
 	\]
 	
 	\item If $t \ge t_1$, by the second case, we can
 	bound both
 	\[
 	d(\tilde v(t), p \tilde u(t))
 	\quad \text{and} \quad
 	d(\tilde v(t+1), p \tilde u(t+1))
 	\]
 	by $4\ell$. Then
 	\[
 	d_1(g_{t}\tilde v, p g_t \tilde u) \le 8 \ell.
 	\]
 	
 	\item Let $t_1 -1 < t < t_1$. We have
 	\begin{align*}
 		d_1(g_{t}\tilde v, g_t \tilde u) &= d(\tilde v (t), \tilde u (t)) + d(\tilde v (t+1), \tilde u (t+1)) \\
 		&\le  d(\tilde v (t), \tilde u (t)) + d(\tilde v (t+1), p \tilde u (t+1)) + d( p \tilde u (t+1), \tilde u (t+1)).
 	\end{align*}
	The first and the second terms in the last line are bounded by $3\ell$ and $4 \ell$ by Cases 1 and 2 above, respectively. 
 	
	In order to bound the third term, we write $\tilde u (t+1) =x_1 + e^{b_1} i$. Recalling that $q=\tilde u (t_1)= q_1 +e^b i$, we have
	\[
	0<b-b_1= d(\tilde u (t+1), \tilde H)\le d(\tilde u (t+1), q) \le 1.
	\]
	Then \[
	d(p \tilde u (t+1), \tilde u (t+1)) \le \lambda e^{-b_1} \le \lambda e^{-b +1} = e \,\ell\le 3 \ell.
	\] We conclude that, if $t_1 -1 < t < t_1$, then 
	\[
	d_1(g_{t}\tilde v, g_t \tilde u) \le 10 \ell(\gamma).
	\]

 \end{enumerate}

Finally, we observe that 
\begin{align*}
	d_1\!\left(g_{t+\tau} \tilde v ,\, g_t\tilde u\right) & \le d_1\!\left(g_{t+\tau} \tilde v ,\, g_t\tilde v\right) + d_1\!\left(g_{t} \tilde v ,\, g_t\tilde u\right)\\
	& =  2|\tau| + d_1\!\left(g_{t} \tilde v ,\, g_t\tilde u\right) \\
	&\le  2\ell + d_1\!\left(g_{t} \tilde v ,\, g_t\tilde u\right),
\end{align*}
thanks to Proposition \ref{bound_wind}. Similarly we have
\[
d_1\!\left(g_{t+\tau} \tilde v ,\, p g_t\tilde u\right) \le 2\ell + d_1\!\left(g_{t} \tilde v ,\, p g_t\tilde u\right).
\]
Together with the previous bounds, these imply the statement.
 \end{proof}


\section{Proof of Theorem \ref{thm2}}

\subsection{Cusp-recurrent vectors}

\begin{proposition}\label{cusprecurrent}
	Let $S= \Gamma \backslash \mathbb{H}$ be a hyperbolic surface with at least one cusp. Let $u\in T^1S$ be a cusp-recurrent vector and let $\tilde u$ be a lift to $T^1 \H$. Then there exists a sequence of pairs $(\tilde H_n, p_n)$ such that
	\begin{enumerate}
		\item $\tilde u$ is tangent to the oriented pair $(\tilde H_n, p_n)$,
		\item the tangency point is $\tilde u(t_n)$, where $t_n$ form an increasing sequence of non-negative numbers tending to $+\infty$,
		\item the fixed points $x_n$ of the isometry $p_n$ are distinct and form a monotonic sequence converging to $\tilde u(+\infty)$,
		\item the horoballs bounded by $\tilde H_n$ are pairwise disjoint,
		\item $\ell(\tilde H_n , p_n)$ tends to $0$.
	\end{enumerate}
\end{proposition}

\begin{proof}
	By definition, there exists a sequence of closed horocycles $(H_n)_{n\in \N}$ in $S$ with lengths $\ell_n$ converging to $0$ and a sequence of non-negative times $(t_n)_{n\in \N}$ such that $u(t_n)$ belongs to $H_n$.
	
	
	Then, for every $n\in N$, $\tilde u(\reals^+)$ intersects a lift $\tilde H_n^0$ of the horocycle $H_n$ at the point $q_n^0:=\tilde u (t_n^0)$. There exists an isometry $p_n$ preserving the horocycle $\tilde H_n^0$ and such that $\ell(\tilde H_n^0 , p_n) = \ell_n$. Let $x_n$ denote the fixed point of $p_n$ in $\partial \H$, which is also the center of $\tilde H_n^0$.
	
	If a vector points to a parabolic fixed point, then its associated geodesic ray in the quotient diverges through a cusp. We deduce that $x_n$ is different from $\tilde u(+\infty)$, since $u(\reals^+)$ does not diverge through a cusp. Moreover, we next show that $x_n$ cannot take the same value infinitely often. Assume that, after taking a subsequence, $x_n=\xi \in \partial \H$ for all $n\in \N$. The stabilizer of $\xi$ in $\Gamma$ is generated by a parabolic isometry $p$ and $p_n$ is either $p$ or $p^{-1}$. Since $\ell (\tilde H_n^0 , p_n)$ tends to $0$, then $\tilde H_n^0$ is a sequence of horocycles centered at $\xi$ and converging to $\xi$ when $n\to +\infty$. Since $q_n^0\in \tilde H_n^0$, this implies that the geodesic ray generated by $\tilde u$ points to $\xi$, obtaining a contradiction.
	
	Up to considering a subsequence, we can assume that $x_n$ is also different from $\tilde u(-\infty)$. For every $n$, $\tilde u(\reals)$ is a geodesic intersecting at least once $\tilde H_n^0$ whose endpoints do not coincide with the center of $\tilde H_n^0$. Then, there exists a unique horocycle $\tilde H_n$ in the closed horoball bounded by $\tilde H_n^0$ which is tangent to the ray $\tilde u (\reals )$. Let $q_n=\tilde u (t_n)$, $t_n\in \reals$, be the tangency point. Moreover, $\ell(\tilde H_n,p_n)\le\ell(\tilde H_n^0,p_n)=\ell_n$.  
	
	The injectivity radius of $S$ at the point $u(t_n)$ is bounded above by $\ell(\tilde H_n,p_n)$, which tends to $0$. This implies that $t_n$ must tend to $+\infty$. Let us show that $t_n$ is positive for all but finitely many $n$. If that is not the case, up to taking a subsequence, assume that $t_n \le 0$ for all $n$. Then $\tilde u (0)$ belongs to the closed horoball bounded by $\tilde H_n^0$. This implies that the injectivity radius at $u(0)$ is less than $\ell(\tilde H_n^0,p_n)$ for all $n$, which is a contradiction. Therefore, up to considering a subsequence, we can assume that $(t_n)_n$ is an increasing sequence of nonnegative times that tends to $+\infty$.
	
	Up to applying an isometry of $\mathbb{H}$ we can assume that $\tilde u(0)=i$ and $\tilde u(+\infty) =\infty$. In the half-plane model, $\tilde H_n$ is an euclidean circle tangent to both the real and the imaginary axis at the points $x_n$ and $\tilde u(t_n)$. Up to taking the reflection with respect to the real axis and a subsequence, we assume that the horocycles $\tilde H_n$ are in the first quadrant. Since $t_n$ is increasing and goes to infinity and $x_n=e^{t_n}$, then $x_n$ is an increasing sequence tending to $+\infty$. 
	
	Up to extracting subsequences, we can also assume that the horoballs bounded by $\tilde H_n$ are pairwise disjoint. Finally, up to changing $p_n$ by its inverse, we can assume that $\tilde u$ is tangent to the each oriented pair $(\tilde H_n,p_n)$. In Figure \ref{pic3} we represent the ray of $\tilde u$ and the horocycles $\tilde H_n$.
\end{proof}

\begin{figure}[h]
	\centering
	\includegraphics[scale=1]{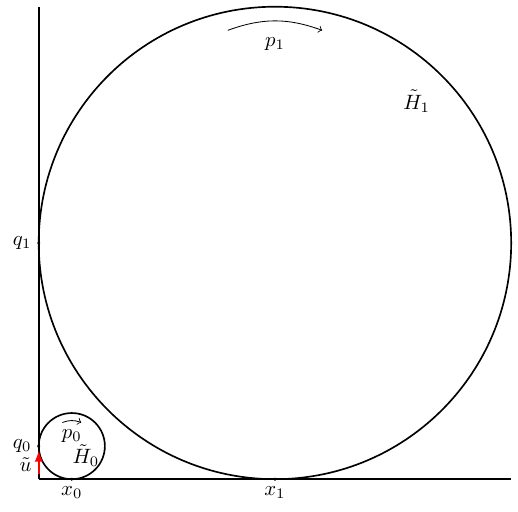}
	\caption{Position of the sequence of horocycles $\tilde H_n$.} 
	\label{pic3}
\end{figure}

From now on, we fix a cusp-recurrent vector $u\in T^1 S$. Up to conjugating the group $\Gamma$, we assume that $\tilde u$ is a lift of $u$ satisfying $\tilde u(0) =i$ and $\tilde u(+\infty)=\infty$. We also fix a sequence of pairs $(\tilde H_n , p_n)$ given by Proposition \ref{cusprecurrent}.

Let $\alpha= (\alpha_n)_{n\in \N}$ be a subsequence of $(p_n)_{n\in \N}$. More precisely, let $k_n$ be the increasing sequence of nonnegative integers such that $\alpha_n = p_{k_n}$. Then we denote
\[
x_n^\alpha = x_{k_n}, \, \tilde H_n^{\alpha} =\tilde H_{k_n} \text{ and } \ell^{\alpha}_n =\ell(\tilde H_{k_n}, p_{k_n}).
\]

\subsection{Construction of $w_\alpha$ in $\overline{H^{ss} u }$}

\par 	
	Let $(\alpha_n)_{n\in\\N}$ be a subsequence of $(p_n)_{n\in \N}.$ We introduce $\beta_n=\alpha_0\ldots\alpha_n\in\Gamma$ and $\tilde v_n\in T^1\H$ defined by:
	\begin{itemize}
		\item $\tilde v_n(0)=i,$
		\item $\tilde v_n(+\infty)=\beta_n(\infty)$.
	\end{itemize}

\begin{lemma}\label{lemma_vectors}
	The sequence $(\tilde v_n)_{n\in\N}$ converges towards a vector $\tilde v_\alpha\in T^1 \H $ defined by $\tilde v_\alpha (0)=i$ and $\tilde v_\alpha(+\infty)=\lim\limits_{n\to +\infty}\beta_n(\infty)$.
	
\end{lemma}

\begin{proof}
	We have to prove that the sequence $(\beta_n(\infty))_{n\in\N}\subset \partial \H$ converges. Let $x_n^{\alpha}>0$ denote the fixed point of $\alpha_n$. 
	Using the dynamics of $\alpha_n$ we have
	\begin{itemize}
		\item $\forall n\geq 0$, $x_n^{\alpha} < \alpha_n\infty$,
		\item $\forall n\geq 1$, $x_{n-1}^{\alpha}<\alpha_{n-1}\alpha_n\infty<\alpha_{n-1}\infty$,
		\item $\forall i\geq 0,$ if $x_i^{\alpha}<x<y$, then $x_i^{\alpha}<\alpha_i x<\alpha_iy.$
	\end{itemize}
	It follows that $(\beta_n \infty )_{n\in\N}$ is a decreasing sequence of positive numbers and hence converges.

 \end{proof}



We observe that $\tilde v_0$ is the winding of $\tilde u$ around the pair $(\tilde H^{\alpha}_0, \alpha_0)$. We next explore the relation of the vectors $\tilde v_n$, $n\ge 1$, with the notion of winding.

\begin{lemma}
	For every $n\ge 0$, the geodesic ray generated by $\beta_n^{-1} \tilde v_n$ intersects the horocycle $\tilde H_{n+1}^{\alpha}$.
\end{lemma}

\begin{proof}
For $n\ge 0$, $\tilde H_{n}^{\alpha}$ is an euclidean circle in the first quadrant tangent to the real and the imaginary axis at the points $x^\alpha_n$ and $q^\alpha_n$, respectively. Let $R_n$ be the closed region subtended by the geodesic ray between $q_n^\alpha$ and $0$ and the horocyclic arc of $\tilde H_{n}^{\alpha}$ between $q^\alpha_n$ and $x^\alpha_n$ in the first quadrant (see Figure \ref{pic4}). We notice that $R_n\subset R_{n+1}$ for all $n\ge 0$. The dynamics of $\alpha_n$ also imply that $\alpha_n^{-1} (R_n) \subset R_n$.

\begin{figure}[h]
	\centering
	\includegraphics[scale=1]{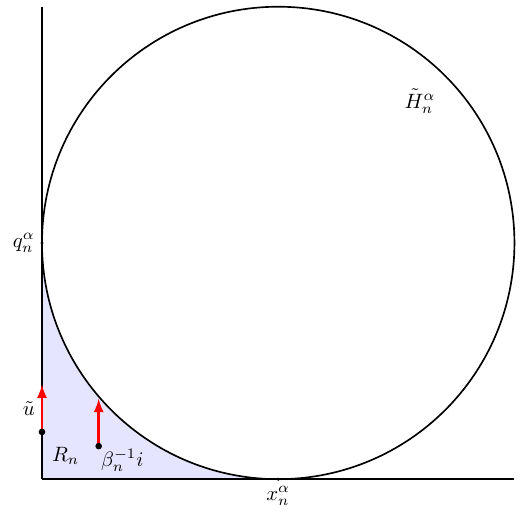}
	\caption{Region $R_n$ and ray of $\beta_n^{-1} \tilde v_n $.} 
	\label{pic4}
\end{figure}

By induction, one proves that $\beta_n^{-1} i \in R_n$ for all $n\ge 0$. Since the geodesic ray generated by $\beta_n^{-1} \tilde v_n$ joins $\beta_n^{-1} i$ to $\infty$, then it is a vertical half-line that intersects $\tilde H_{n+1}^{\alpha}$. 
\end{proof}

By the previous lemma, for every $n\ge 0$, there exists a unique horocycle $\tilde H_{n+1}'$ centered at $x_{n+1}^\alpha$ that is tangent to the geodesic ray $\beta_n^{-1} \tilde v_n (\reals^+)$ and
\begin{equation}\label{shrink_horo}
	\ell (\tilde H_{n+1}', \alpha_{n+1}) \le \ell (\tilde H_{n+1}^{\alpha}, \alpha_{n+1}).
\end{equation}

For $n\ge 0$, we can consider the winding of $\beta_n^ {-1}\tilde v_n$ around the pair $(\tilde H_{n+1}',\alpha_{n+1})$. By definition $\mathrm{Wind}_{( \tilde H'_{n+1}, \alpha_{n+1})}(\beta_n^{-1} \tilde v_n)$ is a vector based at $\beta_n ^{-1} i $ pointing at $\alpha_{n+1} \infty$. This shows that, for all $n\geq 0$,
$$
\tilde v_{n+1}=\beta_n \mathrm{Wind}_{(\tilde  H'_{n+1}, \alpha_{n+1})}(\beta_n^{-1} \tilde v_n).
$$

According to Proposition \ref{bound_wind}
\begin{equation}\label{winding_time_seq}
	\tau_{\tilde H_{n+1}', \alpha_{n+1}, \beta_{n}^{-1}\tilde v_n} = B_{\infty }(\alpha_{n+1}^{-1} \beta_n^{-1} i , \beta_n^{-1} i).
\end{equation}


We define \[
r_n := B_{\infty}(\beta_n^{-1} i , i ).
\]
We notice that $ g_{r_n} \tilde v_n \in \beta_n H^{ss} \tilde u$.

\begin{proposition}\label{convergence_winding_time}
	Let $(\alpha_n)_{n\in\\N}$ be a subsequence of $(p_n)_{n\in \N}$ such that \[ \sum_{n=0}^{+\infty} \ell_n^{\alpha} < +\infty.\] Then the sequence $(r_n)_{n\in \mathbb{N}}$ converges towards a real number $r_{\alpha}$. 
	Moreover, $|r_n|\le \sum_{i=0}^{+\infty} \ell_i^{\alpha}$ for all $n\ge 0$ and, hence, $|r_{\alpha}|\le \sum_{i=0}^{+\infty} \ell_i^{\alpha}$ 
\end{proposition}

\begin{proof}
	We have \[\left | r_{n+1}-r_n  \right |=\left | B_\infty(\alpha_{n+1}^{-1}\beta_n^{-1}(i), \beta_n^{-1}(i)) \right |  = |\tau_{\tilde H_{n+1}', \alpha_{n+1}, \beta_{n}^{-1}\tilde v_n} | \] by \ref{winding_time_seq}. By Proposition \ref{bound_wind} and Equation \ref{shrink_horo},
	it follows that
	$$
	\left | r_{n+1}-r_n  \right | \le  \ell(\tilde H_{n+1}', \alpha_{n+1})\leq \ell(\tilde H_{n+1}^{\alpha}, \alpha_{n+1}) = \ell^{\alpha}_{n+1}.$$
	Hence \[
	\left | r_{m}-r_n  \right |\le \sum_{i=n+1}^{m}\ell^{\alpha}_{i}.
	\]
	As a consequence $(r_n)_{n\in \N}$ is a Cauchy sequence and hence converges.
	
	Moreover, \[
	|r_n|\le |r_0| + \sum_{i=0}^{n-1} \left | r_{i+1}-r_i  \right | \le \sum_{i=0}^{n-1}\ell^{\alpha}_{i} \le  \sum_{i=0}^{+\infty}\ell^{\alpha}_{i}.
	\]
	Taking the limit when $n\to +\infty$ we conclude the last statement.
\end{proof}

For a subsequence $(\alpha_n)_{n\in\\N}$ of $(p_n)_{n\in \N}$ satisfying the assumptions of Proposition \ref{convergence_winding_time}, we define $\tilde w_\alpha$ as 
$$\tilde w_\alpha=g_{r_\alpha}\tilde v_\alpha=\lim\limits_{n\to \infty}g_{r_n}\tilde v_n.$$


\subsection{Construction of $w^\alpha$ in $W^{ss} u$}


From now on, we work on $T^1 S$. We let $v_n, v_{\alpha}, w_{\alpha}$ be the projections of $\tilde v_n, \tilde v_{\alpha}, \tilde w_{\alpha}$, respectively. The next result provides uniform bounds of the distance between the geodesic orbits of $v_n$ and $u$.

\begin{proposition}
	Let $\varepsilon>0$ and let $(\alpha_n)$ be a subsequence of $(p_n)$ such that
	\[
	\ell_n^\alpha < \frac{1}{12}\frac{\varepsilon}{4^n}.
	\]
	Then, for every $l\in \N$, there exists a time $T_l \ge 0$ such that, for all
	$n \in \mathbb{N}$,
	\begin{equation}\label{Pm}
			\forall t\ge T_l,\quad d_1\big(g_{t+r_n} v_n,\, g_t u\big)
		<
		\left(\sum_{k=0}^{n} \frac{1}{2^k}\right)\frac{\varepsilon}{2^l}.
	\end{equation}
\end{proposition}

\begin{proof}
	We first observe that $\sum_{i=0}^{+\infty}\ell^{\alpha}_{i}\le \varepsilon/9$, so we are in the hypothesis of Proposition \ref{convergence_winding_time}.
	
	We start by defining the times $T_l$. 
	Recall that $\beta_n^{-1}\tilde v_n$ is a vector pointing at $\infty$
	and $\beta_n^{-1} g_{r_n}\tilde v_n \in H^{ss} \tilde u$.
	Let
	\[
	D_n = \max_{0 \le k \le n}
	d\big(\beta_k^{-1} g_{r_k}\tilde v_k(0),\, i\big).
	\]
	
	We now set $T_0 = \varepsilon/ 9$ and, for $l \ge 1$,
	\[
	T_l
	= \max \left(
	\log\!\left(
	\sinh\!\left(\frac{D_{l-1}}{2}\right)\cdot 2^{l+2} \cdot \varepsilon^{-1}
	\right), \frac{\varepsilon}{9} \right)
	\]
	
	We will show that this sequence of times satisfies the property \ref{Pm}
	of the statement.
	The proof is by induction. We say that the property $P_m$ is satisfied
	if \ref{Pm} is satisfied for all $n,l \in \{0,\dots,m\}$.
	The statement follows if we prove $P_m$ for all $m$.
	
	We start by showing that $P_0$ is satisfied.
	Recall that $\tilde v_0 = \mathrm{Wind}_{(\tilde H^{\alpha}_0, \alpha_0)}(\tilde u)$. By Proposition \ref{key_prop}, then we have, for all $t\ge 0$
	\[
	d_1(g_{t+r_0}v_0, g_t u) \le  12 \ell^{\alpha}_0 < \varepsilon.
	\]

	We now assume that $P_m$, $m\ge0$, is satisfied and show $P_{m+1}$. We need to show \ref{Pm} for $n,l\in \{0,\dots,m+1\}$. If we take $n,l\in \{0,\dots,m\}$ then \ref{Pm} is satisfied by the induction assumption.
	
	Let us first consider the case $n \in \{0,\dots,m\}$ and $l=m+1$. Observe that $g_{r_n} \beta_n^{-1}\tilde v_n$ belongs to the horocycle $H^{ss}\tilde u$. Then we can estimate the distance $d_1$ by the distance between the basepoints and use Lemma \ref{distance_horo}, after applying the inequality $x\le \sinh(x)$ for $x\ge0$. From the definition of the time $T_{m+1}$ we obtain that, for all $t\ge T_{m+1}$, 
	\begin{align*}
		d_1\big(g_{t+r_n} \beta_n^{-1} \tilde v_n,\, g_t \tilde u\big) \le & 
		2 d\big( \beta_n^{-1} \tilde v_n (t+r_n),\, \tilde u (t) \big) \le 
		4 \sinh (\frac{d( \beta_n^{-1} \tilde v_n (t+r_n),\, \tilde u (t) )}{2}) \\
		\le & 4 \sinh (\frac{d( \beta_n^{-1} \tilde v_n (r_n),\, \tilde u (0) )}{2}) e^{-t} \le 4 \sinh(\frac{D_m}{2}) e^{-t}	\le \frac{\varepsilon}{2^{m+1}}.	
	\end{align*}
Hence, for all $t\ge T_{m+1}$,
\[
d_1\big(g_{t+r_n} v_n,\, g_t u\big) \le \frac{\varepsilon}{2^{m+1}} \le \left(\sum_{k=0}^{n} \frac{1}{2^k}\right)\frac{\varepsilon}{2^{m+1}}.
\]
So far we have proved that \ref{Pm} is verified for any $n\in\{0,\dots,m\}$ and any $l\in \{0,\dots,m+1\}$. 

Finally, let us consider the case that $n=m+1$ and $l\in \{0,\dots,m+1\}$. 
Recall that \[\tilde v_{m+1}=\beta_m \mathrm{Wind}_{(\tilde  H'_{m+1}, \alpha_{m+1})}(\beta_m^{-1} \tilde v_m)\] and $\tau_{\tilde H_{m+1}', \alpha_{m+1}, \beta_{m}^{-1}\tilde v_m} = r_{m+1}-r_m$. Applying Proposition \ref{key_prop}, we have
\[
d_1\big(g_{s+r_{m+1}-r_m} \beta_m^{-1}\tilde v_{m+1},\, g_{s} \beta_m^{-1} \tilde v_m\big) \le 12 \ell^{\alpha}_{m+1} < \frac{\varepsilon}{4^{m+1}}
\]
for all $s\ge 0$. Hence,
\[
d_1\big(g_{t+r_{m+1}} v_{m+1},\, g_{t+r_m} v_m\big) < \frac{\varepsilon}{4^{m+1}}
\]
for all $t\ge r_m$, and in particular this holds if $t\ge \varepsilon/ 9$ thanks to Proposition \ref{convergence_winding_time}. Using \ref{Pm}, that we already established for $n\in\{0,\dots,m\}$ and any $l\in \{0,\dots,m+1\}$, if $t\ge T_l$, we have
\begin{align*}
	d_1\big(g_{t+r_{m+1}} v_{m+1},\, g_t u\big) \le & \, d_1\big(g_{t+r_{m+1}} v_{m+1},\, g_{t+r_m} v_m\big) + d_1\big(g_{t+r_m} v_m,\, g_t u\big) \\
	 <& \, \frac{\varepsilon}{4^{m+1}} + \left(\sum_{k=0}^{m} \frac{1}{2^k}\right)\frac{\varepsilon}{2^l} \le \left(\sum_{k=0}^{m+1} \frac{1}{2^k}\right)\frac{\varepsilon}{2^l}.
\end{align*}

\end{proof}

\begin{corollary}\label{cor1}
	Let $\varepsilon>0$ and $(\alpha_n)$ be a subsequence of $(p_n)$ such that
	\[
	\ell_n^\alpha < \frac{1}{12}\frac{\varepsilon}{4^n}.
	\]
	Then $w_\alpha \in W^{ss} u$.
	
	Moreover, for all $t\ge 0$, $d_1\big(g_{t} w_{\alpha},\, g_t u\big)\le 3\varepsilon$.
\end{corollary}
\begin{proof}
	For a fixed $l\ge 0$, we have, for all $t\ge T_l$,
	\[
	d_1\big(g_{t+r_n} v_n,\, g_t u\big)
	<
	\frac{\varepsilon}{2^{l-1}}.
	\]
	Taking the limit when $n\to +\infty$, we obtain, for all $t\ge T_l$,
	\[
	d_1\big(g_{t} w_{\alpha},\, g_t u\big)
	\le
	\frac{\varepsilon}{2^{l-1}}.
	\]
	This shows that \[\lim_{t\to +\infty} d_1\big(g_{t} w_{\alpha},\, g_t u\big) = 0.\]
	
	Moreover, if $t\ge T_0$, then $d_1\big(g_{t} w_{\alpha},\, g_t u\big)\le 2\varepsilon$. If $0\le t\le T_0=\varepsilon/9$, then
	\begin{align*}
		d_1\big(g_{t} w_{\alpha},\, g_t u\big) \le& \, d_1\big(g_{t} w_{\alpha},\, g_{T_0} w_{\alpha}) + d_1\big(g_{T_0} w_{\alpha},\, g_{T_0} u\big) + d_1\big(g_{T_0} u,\, g_t u\big) \\
		\le&\,  \frac{2\varepsilon}{9} + 2\varepsilon + \frac{2\varepsilon}{9} < 3\varepsilon.
	\end{align*}

\end{proof}

\subsection{Construction of $w_\alpha$ in $W^{ss} u \setminus H^{ss} u$}

Let $\varepsilon>0$. 
We denote by $\Sigma$ the set of subsequences $\alpha = (\alpha_n)_{n\in \N}$ of $(p_n)_{n\in \N}$ such that \[\ell_n^\alpha < \frac{1}{12}\frac{\varepsilon}{4^n}.\]
We define a map $\xi : \Sigma \to \reals$ as follows: for $\alpha = (\alpha_n)_{n\in \N}$, let $\xi(\alpha)$ be the point $\lim_{n\to +\infty} \alpha_n \dots \alpha_0 \infty$ which exists and is a positive number by Lemma \ref{lemma_vectors}. Notice that $\xi(\alpha)$ is nothing else than the point at infinity of $\tilde w_{\alpha}$.

\begin{proposition}
	The image of $\xi$ is uncountable.
\end{proposition}
\begin{proof}
	The proof is by contradiction. Assume that the image of $\xi$ is countable. We can enumerate its elements as $\{\xi_0 , \xi_1, \dots \}$, with $\xi_k$ being the image by $\xi$ of a sequence $\alpha^k = (\alpha^k_n )_{n\in \N}$ in $\Sigma$. Our goal is to find a sequence $\alpha' \in \Sigma$ such that $\xi(\alpha')$ is different from all the $\xi_n$.
	
	Denote $\ell_n= \ell (\tilde H_n,p_n)$. Recall that the fixed points $x_n$ of the parabolic isometries $p_n$ form an increasing sequence of real numbers tending to infinity and that the lengths $\ell_n$ tend to $0$. The choice of $\alpha_n'$ is done inductively as follows:
	
	First, we choose $p_{k_0}$ such that the fixed point $x_{k_0}$ is greater than $\xi_0$ and $ \ell_{k_0} < \frac{\varepsilon}{12}$. We let $\alpha_0'= p_{k_0}$.
	
	Now, assume that we have chosen $\alpha_0'= p_{k_0},\dots , \alpha_n'= p_{k_n}$ that satisfy $\ell_{k_n} < \varepsilon/12 \cdot 4^{-n}$. Now there exists ${k_{n+1}}> k_n$ such that $\ell_{k_{n+1}} < \varepsilon/12 \cdot 4^{-(n+1)}$ and \[
	x_{k_{n+1}} > (\alpha_n')^{-1} \dots (\alpha_0')^{-1} \xi_{n+1}.
	\]
	Let $\alpha_{n+1}' = p_{k_{n+1}}$.
	
	Thus we define a sequence $\alpha' = (\alpha'_n)$ that belongs to $\Sigma$.	
	Let us prove that, for any $n\ge 0$, $\xi(\alpha')$ does not coincide with $\xi_n$. 
	
		The relevant property is that, for any subsequence $\alpha$ of $(p_n)_n$, $\xi(\alpha)= \alpha_0 \alpha_1 \dots \infty$ is greater than the fixed point $x_0^{\alpha}$ of the isometry $\alpha_0$. This follows easily from the properties of Lemma \ref{lemma_vectors}. 
		
	For instance, $\xi (\alpha') > x_0^{\alpha'} = x_{k_0} >\xi_0$ by the choice of $k_0$.
	For $n\ge 1$, we have
	\[
	(\alpha_{n-1}')^{-1} \dots (\alpha_0')^{-1} \xi (\alpha ') = \alpha_n' \alpha_{n+1}' \dots \infty > x_n^{\alpha'} = x_{k_n} > (\alpha_{n-1}')^{-1} \dots (\alpha_0')^{-1} \xi_{n}
	\]
	by the choice of $k_n$. This implies that $\xi (\alpha' ) $ does not coincide with $\xi_n$.
	
\end{proof}

\begin{corollary}\label{cor2}
	There exist a subsequence $(\alpha_n)_{n\in \N}$ in $\Sigma$ such that $w_\alpha \notin H^{ws} u$. Moreover, $W^{ss}u$ is an uncountable union of stable horocycles.
\end{corollary}
\begin{proof}
	If $w\in T^1S$ belongs to $H^{ss}u$, then $\tilde w (+\infty)$ belongs to $\Gamma \tilde u(+\infty)$ for any lifts $\tilde w, \tilde u$. Since the image of $\xi$ is uncountable, it intersects uncountably many orbits of $\Gamma$ in $\partial \H$. Hence, there are uncountably many disjoint $H^{ss} w_\alpha$ in $W^{ss} u$ with $\alpha\in \Sigma$.
\end{proof}

\section{Abundance of cusp-recurrent vectors}

\begin{proposition}\label{denseimpliescr}
	Let $S$ be a nonelementary hyperbolic surface with at least one cusp. Then every vector $u\in T^1 S$ whose forward orbit is dense in the nonwandering $\mathrm{NW} (g_t)$ set of $g_t$ is cusp-recurrent. 
\end{proposition}

\begin{proof}
	Let $u\in T^1 S$ be such that $g_{\reals^+} u$ is dense in the nonwandering set of $g_t$. In particular, $u(\reals^+)$ is not divergent.
	
	Let $C$ be a cusp of $S$ and let $v\in T^1 S$ a nonwandering vector such that the geodesic ray $v(\reals^+)$ is included in $C$. Clearly, $g_t v$ is also nonwandering for all $t$.  Since $g_{\reals^+} u$ is dense in the nonwandering set, there exists an increasing sequence of times $t_n\to +\infty$ such that $g_{t_n} u$ is at distance less than $1$ from $g_n v $. 
	
	Observe that $v(n)$ belongs to a closed horocycle around the cusp $C$ whose length converges to $0$.  Since $d(v(n), u(t_n))\le1$ for all $n\in \N$, then the length of the closed horocycle around $C$ passing through $u(t_n)$ also tends to $0$. This shows that $u$ is cusp-recurrent.
	
\end{proof}

It is well known that the geodesic flow $g_t$ on a nonelementary hyperbolic surface $S$ is topologically transitive on its nonwandering set $\mathrm{NW} (g_t)$ i.e. it has a dense orbit in $\mathrm{NW} (g_t)$. This implies that there exists a $G_\delta$ dense subset in $\mathrm{NW} (g_t)$ of vectors whose forward $g_t$-orbit is dense in $\mathrm{NW} (g_t)$, and hence, are cusp-recurrent. 

%

\begin{proposition}\label{finitevol}
	Let $S$ be a noncompact finite volume hyperbolic surface. Then almost every vector in $T^1 S$ is cusp-recurrent.
\end{proposition}

\begin{proof}
	The surface $S$ must have at least a cusp. Since the volume measure is ergodic under the geodesic flow by a classical result of Hopf \cite{Hopf1}, the Birkhoff average of almost every vector converges to the volume measure. In particular, the forward $g_t$-orbit of almost every vector is dense in $T^1S$, hence cusp-recurrent by Proposition \ref{denseimpliescr}.
\end{proof}

Proposition \ref{finitevol} together with Theorem \ref{thm2} imply Theorem \ref{thm3}.

\section{Non-expansiveness and proof of Theorem \ref{thm1}}


We give a definition of expansive flow which allows the constant of expansiveness $\delta$ to depend on the point. Remark that this notion of expansiveness is weaker than the notions that appear in the literature \cite{KatokHasselblatt1995,BowenWalters1972}.

\begin{definition}
	A continuous flow $f_t$ on a metric space $(X,d)$ is \emph{expansive} if for every $x\in X$ there exists a constant $\delta>0$ such that for every $y\in X$ and every continuous function $\phi: \reals \to \reals$ with $\phi(0) =0$,
	\[
	d(f_t(x),f_{\phi(t)} (y)) <\delta, \quad \forall t \in \reals
	\]
	implies that $y=f_{\tau} (x)$ for some $\tau\in \reals$.
\end{definition}

\begin{definition}
	Let $u,w\in T^1 \H$ such that $w(+\infty)\not = u(-\infty)$. We define the local product $[w,u]$ as the unique vector $z$ in $H^{ss}(w)$ such that $z(-\infty)= u(-\infty)$. 
\end{definition}

In the following, we will use the continuity of the product structure.

\begin{lemma}\label{cont_localprod} \cite{KatokHasselblatt1995}
	The local product $(w,u)\mapsto [w,u]$ is uniformly continuous in the distance $d_1$. 
\end{lemma}

\begin{proposition}\label{distancelp}
	For every $\delta>0$ there exists $\varepsilon>0$ such that for every pair $u,w \in T^1 \H$ such that $d_1(u,w) <\varepsilon $ then 
	\begin{align*}
		d_1(g_t w , g_t [w,u] ) <\delta, \quad \forall t\ge 0 ,\\
		d_1(g_t u , g_t [w,u] ) <\delta, \quad \forall t\le 0 . 
	\end{align*}
\end{proposition}

\begin{proof}
	Let $\delta>0$. 
	By Lemma \ref{cont_localprod}, there exists $\varepsilon_1>0$ such that if $d_1(w,u)<\varepsilon_1$ then $d_1(w, [w,u]) = d_1([w,w], [w,u])<\delta$. A similar argument shows that there exists $\varepsilon_2>0$ such that if $d_1(w,u)<\varepsilon_2$ then $d_1(u, [w,u]) <\delta$. Take $\varepsilon=\min\{\varepsilon_1,\varepsilon_2\}$.
	
	Finally, observe that $w$ and $[w,u]$ point to the same point at infinity. An application of Lemma \ref{decreasing}, yields 
	\[
	d_1(g_t w , g_t [w,u] ) \le d_1( w , [w,u] ), \quad \forall t\ge 0.
	\]
	Similarly, we obtain
	\[
	d_1(g_t u , g_t [w,u] ) \le d_1( u , [w,u] ), \quad  \forall t\le 0.
	\]
	These imply the statement.
\end{proof}

%

We first show that, by winding a geodesic along a single short closed horocycle, one can produce a pair of geodesics that remain close for all time, contradicting the expansiveness of the geodesic flow.
\begin{proposition}\label{nonexpansiveness}
	Let $S$ be a nonelementary hyperbolic surface with at least one cusp. Then the geodesic flow $g_t$ on $T^1 S$ is not expansive.
\end{proposition}
	\begin{proof}
		Take $u\in T^1S$ a cusp-recurrent vector such that $u(\reals^-)$ does not diverge through a cusp. Let $\delta>0$ and consider $\varepsilon>0$ given by Proposition \ref{distancelp}. We can assume $\varepsilon\le \delta$. 
		
		Any lift $\tilde u\in T^1 \H$ of $u$ is tangent to an oriented pair $(\tilde H, p)$ with $\ell (\tilde H, p)<\varepsilon/12$. Denote $\tilde v = \mathrm{Wind}_{(\tilde H, p )}(\tilde u)$, $\tau =\tau_{\tilde H, p ,\tilde u}$, $\tilde w = g_{\tau} \tilde v$ and let $w$ be the projection of $\tilde w$ to $T^1S$. By Proposition \ref{bound_wind}, we have 
		\[
		\min \{ d_1 (g_{t} \tilde w ,\, g_t\tilde u ),
		d_1(g_{t} \tilde w,\, p g_t\tilde u)
		\}
		\\ < \varepsilon, \quad \forall t\ge 0.
		\]
		Consider $\tilde y=[\tilde w , \tilde u]$. By Proposition \ref{distancelp}, we obtain
		\begin{align*}
			d_1(g_t \tilde w , g_t \tilde y ) <\delta, \quad \forall t\ge 0 ,\\
			d_1(g_t \tilde u , g_t \tilde y ) <\delta, \quad \forall t\le 0 . 
		\end{align*}
		Let $y$ be the projection of $\tilde y$ to $T^1S$. The previous estimates pass to the quotient and, after combining them, we obtain
		\[
		d_1(g_t u, g_t y) < 2\delta,\quad \forall t\in \reals.
		\]
		We have to show that the $g_t$-orbits of $u$ and $y$ are different. If this were not the case, there would be $\gamma\in \Gamma$ such that $\gamma \tilde u(+\infty) = \tilde y(+\infty )$ and $\gamma \tilde u(-\infty) = \tilde y(-\infty )$. Observe that $\tilde y (+\infty ) =\tilde w (+\infty ) = p \tilde u (+\infty)$ and $ \tilde y (-\infty)= \tilde u (-\infty)$. Then $\gamma^{-1} p $ fixes $\tilde u (+\infty)$. Since $u$ is cusp-recurrent, $\tilde u (+\infty)$ is only fixed by the identity, therefore $\gamma=p$. We deduce that $p$ fixes $\tilde u(-\infty)$, which is not true by hypothesis. Thus the assumption is false: $u$ and $y$ belong to distinct $g_t$-orbits.
		\end{proof}

We now show that we can find uncountably many distinct orbits contradicting the expansiveness, as an application of Theorem \ref{thm2}.

\begin{myproof}{Theorem}{\ref{thm1}} 
	Let $u\in T^1 S$ be a cusp-recurrent vector. 
	Let $\delta>0$. Our goal is to construct uncountably many vectors $v\in T^1 S$ in distinct $g_t$-orbits such that $d_1(g_t u, g_t v)< 2 \delta$ for all $t \in \reals$.
	
	Let $\varepsilon>0$ be the one given by Proposition \ref{distancelp} depending on $\delta>0$. Up to reducing $\varepsilon$, we can assume that $\varepsilon< \delta$. 
	
	By Corollaries \ref{cor1} and \ref{cor2}, there exists an uncountable set $\{w_i\}_{i\in I}$ of vectors $w_i\in W^{ss}u$ such that
	\begin{itemize}
		\item for all $i,j\in I, \, i\not = j $,
		\[
		\Gamma \tilde w_i (+\infty ) \cap \Gamma \tilde w_j (+\infty ) = \emptyset,
		\]
		for any lifts $\tilde w_i, \tilde w_j$ of $w_i, w_j$, respectively,
		
		\item and for all $i\in I$,
		\begin{equation}\label{distance_wu}
			d_1(g_ t w_i, g_t u)<\varepsilon, \quad \forall t\ge 0.
		\end{equation}
		
	\end{itemize}
	 
	Let $\tilde u$ be a lift of $u$. For each $i\in I$, let $\tilde w_i $ be a lift of $w_i$ such that $d_1(\tilde w_i, \tilde u)= d_1(w_i,u) <\varepsilon$. Let $\tilde v_i = [\tilde w_i , \tilde u]$. By Proposition \ref{distancelp}, we have 
	\begin{align*}
		d_1(g_t \tilde w_i , g_t \tilde v_i ) <\delta, \quad \forall t\ge 0 ,\\
		d_1(g_t \tilde u , g_t \tilde v_i ) <\delta, \quad \forall t\le 0 . 
	\end{align*}
	Denoting $v_i$ the projection of $\tilde v_i$ to $T^1 S$, the previous estimates pass to the quotient,
	\begin{align*}
		d_1(g_t w_i , g_t v_i ) <\delta, \quad \forall t\ge 0 ,\\
		d_1(g_t u , g_t v_i ) <\delta, \quad \forall t\le 0 . 
	\end{align*}
	Now, by the previous equations and \ref{distance_wu}, we obtain
	\[
	d_1(g_t u , g_t v_i ) \le d_1(g_t u , g_t w_i ) + d_1(g_t w_i , g_t v_i ) < \varepsilon + \delta \le 2\delta
	\]
	for all $t\ge 0$, and hence for all $t\in \reals$.
	
	Given $i,j\in I, \, i\not =j$, $\tilde v_i(+\infty) =\tilde w_i(+\infty)$ and $\tilde v_j(+\infty) =\tilde w_j(+\infty)$ are in distinct $\Gamma$-orbits, therefore $v_i$ and $v_j$ belong to distinct $g_t$ orbits in $T^1 S$.
	We have constructed an uncountable set $\{v_i\}_{i\in I}$ satisfying the desired properties.
\end{myproof}

\bibliographystyle{alpha}
\bibliography{general_library}

\end{document}